\documentclass{psp}
\usepackage{graphicx}
\usepackage{amsmath}
\usepackage{amssymb}
\usepackage{amsfonts}
\usepackage{verbatim}
\usepackage{scalefnt}
\usepackage{multirow}
\usepackage{mathtools}
\usepackage{bbm}
\usepackage{float}
\usepackage{bm}
\usepackage{booktabs}
\restylefloat{figure}

\begin{document}

\def\COMMENT#1{}
\let\COMMENT=\footnote% COMMENT OUT for clean output

\newtheorem{theorem}{Theorem}
\newtheorem{lemma}[theorem]{Lemma}
\newtheorem{corollary}[theorem]{Corollary}
\newtheorem{proposition}[theorem]{Proposition}
\newtheorem{conjecture}[theorem]{Conjecture}
\newtheorem{definition}[theorem]{Definition}
\newtheorem{claim}[theorem]{Claim}
\newtheorem{optproblem}[theorem]{Problem}

\newenvironment{claimproof}[1]{\par\noindent\underline{Proof:}\space#1}{\leavevmode\unskip\penalty9999 \hbox{}\nobreak\hfill\quad\hbox{$\blacksquare$}}

\def\eps{{\varepsilon}}
\newcommand{\cP}{\mathcal{P}}
\newcommand{\cT}{\mathcal{T}}
\newcommand{\cL}{\mathcal{L}}
\newcommand{\ex}{\mathbb{E}}
\newcommand{\eul}{e}
\newcommand{\pr}{\mathbb{P}}
\newcommand{\Capa}{\mathop{\mathrm{Cap}}}
\newcommand{\opt}{\mathop{\textsc{opt}}}
\newcommand{\feas}{\mathop{\textsc{feas}}}

\newcommand{\hide}[1]{}
\renewcommand{\chi}{\sigma}
\newcommand{\I}[1]{{\mathbbm #1}}
% Comments for communicating with coauthors

\newif\ifnotesw\noteswtrue
\newcommand{\comm}[1]{\ifnotesw $\blacktriangleright$\ {\sf #1}\ 
  $\blacktriangleleft$ \fi}
%\noteswfalse	% turn off marginal notes for now

\title{THE ERD\H{O}S-ROTHSCHILD PROBLEM ON EDGE-COLOURINGS WITH FORBIDDEN MONOCHROMATIC CLIQUES}
\title[Edge-colourings with forbidden monochromatic cliques]{The Erd\H{o}s-Rothschild problem on edge-colourings with forbidden monochromatic cliques}

\author[O.~Pikhurko, K.~Staden and Z.~B.~Yilma]{OLEG PIKHURKO\thanks{O.P.\ was supported by ERC
grant~306493 and EPSRC grant~EP/K012045/1.}\enskip \nobreakand\ KATHERINE STADEN\thanks{K.S.\ was supported by ERC
grant~306493.}\\
  Mathematics Institute and DIMAP, University of Warwick, Coventry, CV\textup{4} \textup{7}AL\addressbreak
  e-mail\textup{: \texttt{$\lbrace$o.pikhurko,k.l.staden$\rbrace$@warwick.ac.uk}}
  \nextauthor ZELEALEM B.~YILMA\\
  Carnegie Mellon University Qatar, Doha, Qatar\addressbreak
  e-mail\textup{: \texttt{zyilma@qatar.cmu.edu}}}

\receivedline{Received \textup{18} May \textup{2016}}

\maketitle

\begin{abstract}
Let $\bm{k} := (k_1,\ldots,k_s)$ be a sequence of natural numbers.
For a graph $G$, let $F(G;\bm{k})$ denote the number of colourings of the edges of $G$ with colours $1,\dots,s$ such that, for every  $c \in \{1,\dots,s\}$,
the edges of colour $c$ contain no clique of order $k_c$. 
Write $F(n;\bm{k})$ to denote the maximum of $F(G;\bm{k})$ over all graphs $G$ on $n$ vertices.
This problem was first considered by Erd\H{o}s and Rothschild in 1974,
%~\cite{ER,ER2}, 
but it has been solved only for a very small number of non-trivial cases.
%~\cite{abks,PY,yuster}.

We prove that, for every $\bm{k}$ and $n$, there is a complete multipartite graph $G$ on $n$ vertices with $F(G;\bm{k}) = F(n;\bm{k})$.
Also, for every $\bm{k}$ we construct a finite optimisation problem whose maximum is equal to the limit of $\log_2 F(n;\bm{k})/{n\choose 2}$
as $n$ tends to infinity.
Our final result is a stability theorem for complete multipartite graphs $G$, describing the asymptotic structure of such $G$ with $F(G;\bm{k}) = F(n;\bm{k}) \cdot 2^{o(n^2)}$ in terms of solutions to the optimisation problem.
\end{abstract}

\section{Introduction and results}

Let a sequence $\bm{k} = (k_1,\ldots,k_s) \in \mathbb{N}^s$ of natural numbers be given. By an \emph{$s$-edge-colouring} (or \emph{colouring} for brevity)
of a graph $G=(V,E)$ we mean a function $\chi:E\to [s]$, where we denote $[s]:=\{1,\dots,s\}$. Note that we do not require colourings to be proper, that is, adjacent edges can have the same colour. A colouring $\chi$ of $G$ is called $\bm{k}$-\emph{valid} if, for every $c\in [s]$, the colour-$c$ subgraph $\chi^{-1}(c)$ contains no copy of $K_{k_c}$, the complete graph of order $k_c$. 
Write $F(G; \bm{k})$ for the number of $\bm{k}$-valid colourings of $G$.

In this paper, we investigate  $F(n;\bm{k})$, the maximum of $F(G;\bm{k})$ over all graphs $G$ on $n$ vertices, and the \emph{$\bm{k}$-extremal} graphs, i.e.~order-$n$ graphs which attain this maximum. We assume throughout the paper that $s\ge 2$ and that $k_c \geq 3$ for all $c \in [s]$ (since $k_c=2$ just forbids
colour $c$ and the problem reduces to one with $s-1$ colours).

\subsection{Previous work}

The problem, namely the case when $k_1 = \ldots = k_s =: k$, was first considered by Erd\H{o}s and Rothschild in 1974 (see~\cite{ER,ER2}). 
Clearly, any colouring of a $K_{k}$-free graph is $\bm{k}$-valid.
By Tur\'an's theorem~\cite{turan}, the maximum such graph on $n$ vertices is $T_{k-1}(n)$, the complete $(k-1)$-partite graph with parts as equal as possible.
This implies the trivial lower bound
\begin{equation}\label{trivial}
F(n;(k,\dots,k)) \geq s^{t_{k-1}(n)},
\end{equation}
where $t_{k-1}(n)$ is the number of edges in $T_{k-1}(n)$.
In particular, Erd\H{o}s and Rothschild conjectured that, when $\bm{k}=(3,3)$ and $n$ is sufficiently large, the trivial lower bound~(\ref{trivial}) is in fact tight  and, furthermore,  $T_{2}(n)$ is the unique
$\bm{k}$-extremal graph.
The  conjecture was verified for all $n \geq 6$ by Yuster~\cite{yuster} (who also computed $F(n;(3,3))$ for smaller $n$).
Yuster generalised the conjecture to $\bm{k}=(k,k)$ and proved an asymptotic version.
The full conjecture for all $k \geq 3$ was proved by  Alon, Balogh, Keevash and Sudakov~\cite{abks} who further showed that an analogous result holds for three colours:

\begin{theorem}[Alon, Balogh, Keevash and Sudakov~\cite{abks}]\label{alonetal}
Let $k,n \in \mathbb{N}$ where $k \geq 3$ and $n \geq n_0(k)$.
Then
\[
F(n;(k,k)) = 2^{t_{k-1}(n)} \ \text{ and } \ F(n;(k,k,k)) = 3^{t_{k-1}(n)}.
\]
Moreover, $T_{k-1}(n)$ is the unique extremal graph in both cases.\hfill$\square$
\end{theorem}

The proof of Theorem~\ref{alonetal} uses Szemer\'edi's Regularity Lemma.
Unfortunately, this also means that the graphs to which it applies are very large indeed.
In fact, the assertions are not true for all numbers $n$ of vertices.
As was remarked in~\cite{abks}, the conclusion of Theorem~\ref{alonetal} fails when $k \leq n < s^{(k-2)/2}$, as in this case a random colouring of the edges of $K_n$ with $s$ colours contains no monochromatic $K_k$ with probability more than $1/2$.
Thus, for this range of $n$, we have $F(n;(k,\ldots,k)) > s^{\binom{n}{2}}/2 \geq s^{t_{k-1}(n)}$.

The authors of~\cite{abks} noted that when more than three colours are used, the behaviour of $F(n;(k,\ldots,k))$ changes, making its determination both harder and more interesting. Namely, it was shown in~\cite[page~287]{abks} 
 %on Page~287
that if $s \geq 4$ (and $k \geq 3$) then $F(n;(k,\ldots,k))$ is exponentially larger than $s^{t_{k-1}(n)}$. 
In particular, any extremal graph has to contain many copies of $K_k$.
In the case when $\bm{k} = (3,3,3,3)$, they determined $\log F(n;\bm{k})$ asymptotically by showing that
$
F(n;(3,3,3,3)) = (2^{1/8}3^{1/2})^{n^2+o(n^2)}
$,
where $T_4(n)$ achieves the right exponent.
Similarly, they proved that $F(n;(4,4,4,4)) = (3^{8/9})^{n^2 + o(n^2)}$, where $T_9(n)$ achieves the right exponent.
Determining the exact answer in these two cases, the first and third author of this paper proved that, when $n \geq n_0$,
$T_4(n)$ is the unique $(3,3,3,3)$-extremal graph on $n$ vertices, and $T_9(n)$ is the unique $(4,4,4,4)$-extremal graph on $n$ vertices. 

It was also proved in~\cite[Proposition~5.1]{abks} that the limit 
 \begin{equation}\label{eq:lim}
 F(\bm{k}):=\lim_{n\to\infty} \frac{\log_2 F(n;\bm{k})}{n^2/2}
 \end{equation}
 exists (and is positive) when $\bm{k}=(k,\dots,k)$. As it is easy to see, the proof from~\cite{abks} extends to
an arbitrary fixed sequence $\bm{k}$.

Erd\H{o}s and Rothschild also considered the generalisation of the problem, where
one forbids a monochromatic graph $H$ (the same for each colour). In~\cite{abks} the authors showed that the analogue of Theorem~\ref{alonetal} holds
when $H$ is \emph{colour-critical}, that is, the removal of any edge 
from $H$ reduces its chromatic number.
(Note that every clique is colour-critical.)
In a further generalisation, Balogh~\cite{balogh} considered edge-colourings which themselves do not contain a specific colouring of a fixed graph $H$.
Other authors have addressed this question in the cases of forbidden monochromatic matchings, stars, paths, trees and some other graphs in~\cite{hkl2,hkl3}, matchings with a prescribed colour pattern in~\cite{hl}, and rainbow stars in~\cite{hlos}.
Extending work in~\cite{hl}, Benevides, Hoppen and Sampaio considered forbidden cliques with a prescribed colour pattern, and using techniques similar to our own, obtained several results in this direction, including a version of Theorem~\ref{zykov} below.
In~\cite{hkl}, a colouring version of the Erd\H{o}s-Ko-Rado theorem for families of $\ell$-intersecting $r$-element subsets of an $n$-element set was considered; that is, one counts the number of colourings of families of $r$-sets such that every colour class is $\ell$-intersecting.
A so-called `$q$-analogue' was addressed in~\cite{qana}, which considers a colouring version of the Erd\H{o}s-Ko-Rado theorem in the context of vector spaces over a finite field $GF(q)$.

Alon and Yuster~\cite{alonyuster} studied a directed version of the problem, to determine the maximum number of $T$-free orientations of an $n$-vertex graph, where $T$ is a given $k$-vertex tournament. They showed that the answer is $2^{t_{k-1}(n)}$
for $n\ge n_0(k)$.
This in fact answers the original question of Erd\H{o}s~\cite{ER}, which he modified to ask about edge-colourings.

The problem of counting $H$-free edge-colourings in hypergraphs was studied in~\cite{hkl,lprs,lps}.
In an asymptotic hypergraph version of Theorem~\ref{alonetal}, Lefmann, Person and Schacht~\cite{lps} proved that, for every $k$-uniform hypergraph $H$ and $s \in \lbrace 2,3 \rbrace$, the maximum number of $H$-free $s$-edge-colourings over all $k$-uniform hypergraphs with $n$ vertices is $s^{\mathrm{ex}(n,H)+o(n^k)}$, where the Tur\'an function $\mathrm{ex}(n,H)$ is the maximum number of edges in an $H$-free $k$-uniform hypergraph on $n$ vertices.
This is despite the fact that $\mathrm{ex}(n,H)$ is known only for few $H$.

\subsection{New results}

Our first result states that it suffices to consider very special graphs $G$  in order to determine the value of $F(n;\bm{k})$:

\begin{theorem}\label{zykov}
For every $n,s \in \mathbb{N}$ and $\bm{k} \in \mathbb{N}^s$, at least
one of the 
$\bm{k}$-extremal graphs of order $n$ is complete multipartite.
\end{theorem}

Our second result (Theorem~\ref{optred} below) writes the limit in~\eqref{eq:lim} as the value of a certain optimisation problem.
% In order to state some of our later results, we need a few versions of the problem.

\medskip
%\begin{optproblem}\label{ProblemStar}
\noindent\textbf{Problem~$Q_t$}:
 \it
 Given a sequence $\bm{k} := (k_1,\ldots,k_s) \in \mathbb{N}^s$ of natural numbers and $t\in\{0,1,2\}$, determine
\begin{equation}\label{Qdef}
Q_t(\bm{k}) := \max_{(r,\phi,\bm{\alpha}) \in \feas_t(\bm{k})} q(r,\phi,\bm{\alpha}),
\end{equation}
 the maximum value of
 \begin{equation}\label{q}
 q(r,\phi,\bm{\alpha}) := 2\sum_{\stackrel{1 \leq i < j \leq r}{\phi(ij) \neq \emptyset}} \alpha_i \alpha_j \log_2 |\phi(ij)|
 \end{equation}
 over the set $\feas_t(\bm{k})$ of \emph{feasible solutions}, that is, triples
$(r,\phi,\bm{\alpha})$ such that
 %satisfying all of the following constraints:
 \begin{itemize}
  \item $r \in \mathbb{N}$ and $r < R(\bm{k})$, where $R(\bm{k})$ is the \emph{Ramsey number} of $\bm k$ (i.e.\ the minimum $R$ such that $K_R$ admits no $\bm k$-valid
  $s$-edge-colouring);
  \item $\phi \in \Phi_t(r;\bm{k})$, where $\Phi_t(r;\bm{k})$ is the set of all
  functions $\phi : \binom{[r]}{2} \rightarrow 2^{[s]}$
  such that
  \[
  \phi^{-1}(c) := \left\lbrace ij \in \binom{[r]}{2} : c \in \phi(ij) \right\rbrace
  \]
is $K_{k_c}$-free for every colour $c \in [s]$ and $|\phi(ij)|\ge t$
for all $ij\in \binom{[r]}{2}$;
  \item $\bm{\alpha} = (\alpha_1,\ldots,\alpha_r) \in \Delta^r$, where $\Delta^r$ is the set of all $\bm{\alpha} \in \mathbb{R}^r$ with $\alpha_i \geq 0$ for all $i \in [r]$, and $\alpha_1 + \ldots + \alpha_r = 1$.
  %We call $\bm{\alpha}$ a \emph{vertex-weighting}.
  \end{itemize}
 \rm\medskip
 %\end{optproblem}

Note that the maximum in~\eqref{Qdef} is attained. Indeed, for each of the finitely many allowed pairs $(r,\phi)$, the function $q(r,\phi,\cdot)$ is continuous
and hence attains its maximum over the non-empty compact set $\Delta^r$.
A triple $(r,\phi,\bm{\alpha})$ is called \emph{$Q_t$-optimal} if it attains the maximum, that is, 
$(r,\phi,\bm{\alpha})\in \feas_t(\bm{k})$ and $q(r,\phi,\bm{\alpha})=Q_t(\bm{k})$.
%We implicitly assume that $\phi$ has no `coloured twins',~i.e.~that for all $ij \in \binom{[r]}{2}$ such that $\phi(ij) = \emptyset$, there is some $k \in [r] \setminus \lbrace i,j \rbrace$ such that $\phi(ik) \neq \phi(jk)$.

As we will show later in Lemma~\ref{lm:Qs}, $Q_0(\bm{k})=Q_1(\bm{k})=Q_2(\bm{k})$ so we will denote this common value by
$Q(\bm{k})$. Of course, if one wishes to determine the value of $Q(\bm{k})$, then one should work with Problem~$Q_2$ as it has the smallest feasible set.
Since one of our results is stated in terms of $Q_1$-optimal triples (which
may be a strict superset of $Q_2$-optimal triples), we stated different versions of the optimisation problem.
In Section~\ref{conclusion}, we explore how one might hope to solve this optimisation problem, and show that all previously obtained (asymptotic) results can be recovered.

First we show that $Q(\bm{k})$ gives rise to an asymptotic lower bound on $F(n;\bm{k})$.

\begin{lemma}\label{lm:lowerbd}
For every $s \in \mathbb{N}$ and $\bm{k} \in \mathbb{N}^s$, there exists $C$ such that for all $n \in \mathbb{N}$ there is a graph $G$ on $n$ vertices with $F(G;\bm{k}) \geq 2^{Q(\bm{k})\binom{n}{2} - Cn}$.
%Moreover, we can choose $G$ to be complete multipartite.
\end{lemma}

\begin{proof}
Let $(r,\phi,\bm{\alpha})$ be $Q_0$-optimal.
For  $n \in \mathbb{N}$, let $G_{\phi,\bm{\alpha}}(n)$ be the graph of order $n$ with vertex partition $X_1,\ldots,X_r$, where $|\,|X_i| - \alpha_i n| \leq 1$; and in which for all $i,j \in [r]$ and $x_i \in X_i$ and $y_j \in X_j$, we have that $x_iy_j$ is an edge of $G_{\phi,\bm{\alpha}}(n)$ if and only if $i \neq j$ and $\phi(ij) \neq \emptyset$.
Consider those colourings of $G_{\phi,\bm{\alpha}}(n)$ in which $x_iy_j$ is coloured with some colour in $\phi(ij)$, for every $x_i \in X_i$, $y_j \in X_j$, where $1 \leq i < j \leq r$.
Every such colouring is $\bm{k}$-valid because $\phi^{-1}(c)$ is $K_{k_c}$-free for all $c \in [s]$.
The number of such colourings gives the desired lower bound for $F(n;\bm{k})$:
\begin{equation}\label{lowerbd}
F(n;\bm{k}) \geq F(G_{\phi,\bm{\alpha}}(n);\bm{k}) \geq \prod_{\stackrel{1 \leq i < j \leq r}{\phi(ij) \neq \emptyset}}|\phi(ij)|^{|X_i|\,|X_j|}  \geq 2^{Q(\bm{k})\binom{n}{2}-Cn},
\end{equation}
where $C = C(\bm{k})$ is a constant due to rounding.
%Finally, observe that $G_{\phi,\bm{\alpha}}(n)$ is complete $r$-partite if $(r,\phi,\bm{\alpha})$ is $Q_1$-optimal.
\end{proof}

%Note that this asymptotic could be improved; we do not know if there is some constant $C$ such that $F(n;\bm{k}) \leq 2^{C Q(\bm{k})\binom{n}{2}}$.

%The next theorem shows that in fact $F(n;\bm{k})$ is not much larger than the lower bound in Lemma~\ref{lm:lowerbd}.

\begin{theorem}\label{optred}
For every $s\in \mathbb{N}$ and $\bm{k} \in \mathbb{N}^s$, 
we have $F(n;\bm{k})=2^{Q(\bm{k}){n\choose 2}+o(n^2)}$, that is, $F(\bm{k})=Q(\bm{k})$, where $F(\bm{k})$ is the limit in~\eqref{eq:lim}.
%$$\lim_{n \rightarrow \infty} \frac{\log_2 F(n;\bm{k})}{\binom{n}{2}} = Q(\bm{k}).$$
\end{theorem}
%Note that this asymptotic could be improved; we do not know if there is some constant $C$ such that $F(n;\bm{k}) \leq 2^{C Q(\bm{k})\binom{n}{2}}$.

So, as in the result of Lefmann, Person and Schacht~\cite{lps} mentioned above, this theorem can be proved without knowledge of $Q(\bm{k})$.
Our proof of Theorem~\ref{optred} builds upon the techniques of~\cite{abks,PY} and also uses the Regularity Lemma.

The structure of an arbitrary order-$n$ graph $G$ with $F(G;\bm{k})=2^{(Q(\bm{k})+o(1))n^2/2}$
can be rather complicated (see a short discussion in Section~\ref{conclusion} of the case $\bm{k}=(4,3)$). However, the next result states that if $G$ is assumed to be complete multipartite, then the part ratios have to be close to
being $Q_1$-optimal.

\begin{theorem}\label{th:PartiteStructure} For every $\delta>0$ there are $\eta>0$ and $n_0$
such that if $G=(V,E)$ is a complete multipartite graph of order $n\ge n_0$ with (non-empty) parts $V_1,\dots,V_r$ and $F(G;\bm{k})\ge 2^{(Q(\bm{k})-\eta)n^2/2}$ then there is a $Q_1$-optimal triple $(r,\phi,\bm{\alpha'})$ such that the $\ell^1$-distance
between $\bm{\alpha'}\in \Delta^r$ and $\bm{\alpha}=(|V_1|/n,\dots,|V_r|/n)$ is 
at most $\delta$: $\|\bm{\alpha}-\bm{\alpha'}\|_1:=\sum_{i=1}^r|\alpha_i-\alpha_i'|\le \delta$.
\end{theorem}

In a sense, a converse to Theorem~\ref{th:PartiteStructure} holds.
Indeed, for every $Q_1$-optimal triple $(r,\phi,\bm{\alpha}')$, for all $n \in \mathbb{N}$, the proof of Lemma~\ref{lm:lowerbd} gives a complete $r$-partite graph $G_{\phi,\bm{\alpha}'}(n)$ on $n$ vertices with parts $X_1^n,\ldots,X_r^n$ such that, setting $\bm{\alpha}_n = (|X_1^n|/n,\ldots,|X_r^n|/n)$, we have, as $n \to \infty$, that
\[
\frac{\log_2 F(G_{\phi,\bm{\alpha}'}(n);\bm{k})}{n^2/2} \to Q(\bm{k}) \text{ and } \|\bm{\alpha}_n-\bm{\alpha}'\| \to 0.
\]

%\section{Notation}
 %We adopt the convention that $0 \notin \mathbb{N}$.
%Given $r \in \mathbb{N}$, write $[r] := \lbrace 1,\ldots,r \rbrace$.

\medskip
The rest of the paper is organised as follows. Theorem~\ref{zykov} is proved in Section~\ref{symsec}. Section~\ref{lemma} contains a general lemma which is then used in Section~\ref{OptredProof} to prove Theorems~\ref{optred} and~\ref{th:PartiteStructure}.
Section~\ref{conclusion} contains some concluding remarks. We will use the following notation.
For a set $X$ and an integer $k \leq |X|$, let $\binom{X}{k}$ denote the set of all $k$-subsets of $X$. Also, let $2^{X}$ be the set of all subsets of~$X$.
 %Write $\binom{X}{\leq k} := \bigcup_{i=0}^k\binom{X}{i}$.
If it is clear from the context, we may write $ij$ to denote the set $\lbrace i,j \rbrace$ or the ordered pair $(i,j)$.
%We will often write $ij$ to denote the ordered pair $(i,j)$.

\section{Symmetrisation and $\textbf{k}$-extremal graphs}\label{symsec}

%A graph $G$ on $n$ vertices which satisfies $F(n;\bm{k}) = F(G;\bm{k})$ is said to be \emph{$\bm{k}$-extremal}.
In this section we prove Theorem~\ref{zykov}, which states that, for any instance of the problem (i.e.~any choice of the parameters $n,s,\bm{k}$), there is a complete multipartite graph which is $\bm{k}$-extremal.
The proof uses the well-known symmetrisation method that was introduced by Zykov~\cite{zykov}.

\medskip
\noindent
\emph{Proof of Theorem~\ref{zykov}.}
Let $G=(V,E)$ be a $\bm{k}$-extremal graph on $n$ vertices.
Consider distinct vertices $u,v \in V$ with $uv \not \in E$.
Let $G' = G-\{u,v\}$, where $G-X=G[V\setminus X]$ is the graph obtained from $G$ by removing every vertex of a set  $X\subseteq V$ and every edge adjacent to a vertex of $X$.
For a graph $H$, let $\mathcal{F}(H)$ denote the set of $\bm{k}$-valid colourings of $H$. (Thus $F(H;\bm{k})=|\mathcal{F}(H)|$.)
Let $\chi_u$ and $\chi_v$ denote the number of  $\bm{k}$-valid extensions of $\chi \in \mathcal{F}(G')$ to $G-\lbrace v\rbrace$ and $G- \lbrace u \rbrace$ respectively.
Since $uv \notin E$ and each forbidden graph is a clique, we have that the number of  $\bm{k}$-valid extensions of $\chi$ to $G$ is $\chi_u\chi_v$.
Thus
\begin{equation}\label{zykov1}
F(G;\bm{k}) = \sum_{\chi \in \mathcal{F}(G')} \chi_u \chi_v.
\end{equation}
Let $G_u$ be the graph obtained from $G$ by deleting $v$ and adding a new vertex $u'$ which is a clone of $u$ in $G$.
Define $G_v$ analogously.
From (\ref{zykov1}), it follows that 
\begin{equation}\label{zykov2}
F(G_u;\bm{k}) = \sum_{\chi \in \mathcal{F}(G')} \chi_u^2 \ \ \text{ and } \ \ F(G_v;\bm{k}) = \sum_{\chi \in \mathcal{F}(G')} \chi_v^2 .
\end{equation}
Since $G$ is $\bm{k}$-extremal, we have that
\begin{equation}
0 \leq 2F(G;\bm{k}) - F(G_u;\bm{k}) - F(G_v;\bm{k}) \stackrel{(\ref{zykov1}),(\ref{zykov2})}{=} -\sum_{\chi \in \mathcal{F}(G')} (\chi_u - \chi_v)^2 \leq 0,
\end{equation}
and hence we have equality everywhere.
Therefore $G_u$ and $G_v$ are both $\bm{k}$-extremal. In order to finish
the proof, it is enough to show that we can reach a complete multipartite
graph by starting with $G$ and iteratively performing the above operation. 

We say that two vertices $x$ and $y$ are \emph{twins}
(and write $x \sim y$) if they have the same sets of neighbours.
Note that twins are necessarily non-adjacent.
It is easy to see that $\sim$ is an equivalence relation.
Let $[x]_{\sim}$ denote the equivalence class of $x$.
 %, and recall that the equivalence classes of $\sim$ partition~$V$.

%We continue replacing vertices with twins as follows.
Let $G^1 := G$. Repeat the following for as long as possible.
Suppose that we have defined graphs $G^1, \ldots, G^i$ for some $i \geq 1$, which are all $\bm{k}$-extremal.
Suppose that $G^i$ contains a pair $u,v$ of non-adjacent vertices which are not twins.
Choose such a pair so that $|[u]_\sim|$ is maximal.
Let $G^{i+1}=(G^i)_u$ be the graph obtained from $G^i$ by deleting $v$ and adding a new vertex $u'$ which is a clone of $u$.
As was argued above, $G^{i+1}$ is necessarily $\bm{k}$-extremal.

For each $i \geq 1$, call an equivalence class $[x]_{\sim}$ in
the graph $G^i$ \emph{frozen} if $G^i$ is complete between $[x]_{\sim}$ and its complement, and \emph{unfrozen} otherwise.
Let $f(G^i)$ be the sum of sizes of all frozen classes plus the largest size of an unfrozen one.
It is easy to see that $f(G^i)$ is strictly increasing with $i$. Since
$f(G^i)$ is bounded above by $n$, the process terminates in at most $n-1$ steps with some $\bm{k}$-extremal graph $H$.
Since every pair of non-adjacent vertices in $H$ are twins, $H$ is complete multipartite, as desired.
\hfill$\square$

\medskip
Also, the symmetrisation can be applied to $Q_t$-optimal solutions. In particular, one can prove the following.

\begin{lemma}\label{lm:Qs} For every $\bm{k}$, we have $Q_0(\bm{k})=Q_1(\bm{k})=Q_2(\bm{k})$.\end{lemma}

\begin{proof} Since trivially $\feas_0(\bm{k})\supseteq \feas_1(\bm{k})\supseteq \feas_2(\bm{k})$, we have $Q_0(\bm{k})\ge Q_1(\bm{k})\ge Q_2(\bm{k})$.

On the other hand, among all $Q_0$-optimal
solutions $(r,\phi,\bm{\alpha})$, fix one with $r$ as small as possible. 
Then, in particular, we have that each $\alpha_i$ is non-zero.
We claim that necessarily $(r,\phi,\bm{\alpha})\in \feas_2(\bm{k})$ (which will
give the required inequality $Q_2(\bm{k})\ge Q_0(\bm{k})$).
If this is not true, then $|\phi(ij)|\le 1$ for some $ij\in {[r]\choose 2}$, say
for $\{i,j\}=\{r-1,r\}$. For a real $c$, consider $\bm{\alpha'}$ defined
by $\alpha_{r-1}'=\alpha_{r-1}+c$, $\alpha_r'=\alpha_r-c$ and  $\alpha_h':=\alpha_h$ for 
all $h\in[r-2]$. In other words, we shift weight $c$ from $\alpha_r$ to $\alpha_{r-1}$. Since $q(r,\phi, \bm{\alpha'})$ is a linear function $f(c)$ of $c$
and $(r,\phi,\bm{\alpha'})\in \feas_0(\bm{k})$ when $|c|$ is at most
$\min\{\alpha_{r-1},\alpha_r\}>0$, it must be the case that $f(c)$ is
a constant function.
Thus $f(c)=f(0)=Q_0(\bm{k})$ regardless of $c$. In particular,
by taking $c=\alpha_r$, that is, by shifting all weight from $\alpha_r$ to
$\alpha_{r-1}$, we obtain a $Q_0$-optimal solution $(r,\phi,\bm{\alpha'})$ with $\alpha_r'=0$, whose restriction to $[r-1]$
gives another $Q_0$-optimal solution, contradicting the minimality
of $r$.\end{proof}

\section{A unifying lemma}\label{lemma}
%{The asymptotic value of $F(n;\emph{\textbf{k}})$}\label{asymp}\label{opt}
%\subsection{The optimisation problem}\label{opt}

The proofs of Theorems~\ref{optred} and~\ref{th:PartiteStructure} will both follow from the next lemma, which states that the number of $\bm{k}$-valid colourings of any complete $r$-partite graph $H$ can be bounded above by evaluating $q$ for a triple $(r,\phi,\bm{\beta}) \in \feas_1(\bm{k})$, where $\bm{\beta}$ is given by the ratios of the parts of $H$.

\begin{lemma}\label{mainlem}
For all $s \in \mathbb{N}$, $\bm{k} \in \mathbb{N}^s$ and $\eta > 0$, there exists $n_0 \in \mathbb{N}$ such that for every complete multipartite graph $H$ of order $N \geq n_0$ with (non-empty) parts $Y_1,\ldots,Y_r$ with at least one $\bm{k}$-valid colouring, there is some $\phi \in \Phi_1(r;\bm{k})$ such that
\[
\frac{\log_2 F(H;\bm{k})}{N^2/2} \leq q(r,\phi,\bm{\beta}) + \eta,
\]
where $\bm{\beta} := (|Y_1|/N,\ldots,|Y_r|/N)$.
\end{lemma}

In outline, the argument to prove Lemma~\ref{mainlem} is as follows.
The main idea of the proof is to use Szemer\'edi's Regularity Lemma to pass from a $\bm{k}$-valid colouring $\chi$ of $H$ to a set of feasible solutions  
that come from 
$r$-tuples of clusters which are transversal with respect to the
$r$-partition of~$H$. For each obtained solution $(r,\phi,\bm{\beta}) \in \feas_0(\bm{k})$, an upper bound on $q(r,\phi,\bm{\beta})$
can be translated via regularity into an upper bound on the number of restrictions
of possible colourings $\chi$ to the involved clusters (an idea already used in~\cite{abks}).
Then we estimate $F(H;\bm{k})$ by taking an appropriately weighted sum of logarithms of these bounds.
It turns out that the dominant contribution is from those triples $(r,\phi,\bm{\beta})$ that belong to $\feas_1(\bm{k})$, and so the bound obtained for $F(H;\bm{k})$ is in terms of the largest $q(r,\phi,\bm{\beta})$ among such triples.
%This is a generalisation of the idea used at the beginning of the proof of Theorem~\ref{alonetal} in~\cite{abks}.

\subsection{Regularity tools}

We will need the following definitions related to Szemer\'edi's Regularity Lemma.

\begin{definition}[Edge density, $\eps$-regular, $(\eps,\gamma)$-regular, equitable partition]
Given a graph $G$ and disjoint non-empty sets $A,B \subseteq V(G)$, we define the \emph{edge density} between $A$ and $B$ to be
\[
d(A,B) := \frac{|E(G[A,B])|}{|A|\,|B|}.
\]
Given $\eps,\gamma > 0$, the pair $(A,B)$ is called
\begin{itemize}
\item \emph{$\eps$-regular} if for every $X \subseteq A$ and $Y \subseteq B$ with $|X| \geq \eps|A|$ and $|Y| \geq \eps|B|$, we have that $|d(X,Y) - d(A,B)| \leq \eps$;
\item \emph{$(\eps,\gamma)$-regular} if it is $\eps$-regular and has edge density at least $\gamma$.
\end{itemize}
We call a partition $V(G)=V_1\cup\dots\cup V_m$
 \begin{itemize}
 \item \emph{equitable} if $\big|\,|V_i| - |V_j|\,\big| \leq 1$ for all $i,j \in [m]$;
 \item  \emph{$\eps$-regular} if it is equitable,
 $m\ge 1/\eps$, 
 %$|V_i| \leq \eps|V(G)|$ for every $i \in [m]$, 
 and all but at most $\eps \binom{m}{2}$ of the pairs $(V_i,V_j)$ with $1\le i < j\le m$ are $\eps$-regular.
 \end{itemize}
\end{definition}

Our first tool states that an induced subgraph of a regular pair is still regular, provided both parts are not too small.

\begin{proposition}\label{badrefine}
Let $\eps,\delta$ be such that $0 < 2\delta \leq \eps < 1$.
Suppose that $(X,Y)$ is a $\delta$-regular pair, and let $X' \subseteq X$ and $Y' \subseteq Y$.
If
\[
\min \left\lbrace\, \frac{|X'|}{|X|} ,\, \frac{|Y'|}{|Y|}\, \right\rbrace \geq \frac{\delta}{\eps},
\]
then the pair $(X',Y')$ is $\eps$-regular.
\end{proposition}

\begin{proof}
Let $X'' \subseteq X'$ and $Y'' \subseteq Y'$ be such that $|X''| \geq \eps|X'|$ and $|Y''| \geq \eps|Y'|$.
Then $|X''|/|X|, |Y''|/|Y| \geq \delta$.
Since $(X,Y)$ is $\delta$-regular, we have that
$
|d(X'',Y'') - d(X,Y)| \leq \delta.
$
Note further that $|X'|/|X|,|Y'|/|Y| \geq \delta/\eps > \delta$, so
$
|d(X',Y') - d(X,Y)| \leq \delta.
$
By the Triangle Inequality,
$
|d(X'',Y'') - d(X',Y')| \leq 2\delta\le \eps.
$
 This implies that $(X',Y')$ is $\eps$-regular.
\hfill\mbox{ }\end{proof}

We use the following multicolour version of Szemer\'edi's Regularity Lemma~\cite{reg} (see~e.g Theorem~1.18 in Koml\'os and Simonovits~\cite{komsim}).

\begin{lemma}[Multicolour Regularity Lemma]\label{multicol}
For every $\eps > 0$ and $s \in \mathbb{N}$, there exists $M \in \mathbb{N}$ such that for any graph $G$ on $n \geq M$ vertices and any $s$-edge-colouring $\chi : E(G) \rightarrow [s]$, there is an (equitable) partition $V(G) = V_1 \cup \ldots \cup V_m$ with $m \leq M$, which is $\eps$-regular simultaneously with respect to all graphs $(V(G),\chi^{-1}(c))$, with $c \in [s]$.\hfill$\square$
\end{lemma}

Finally, we need the following bound.

\begin{proposition}\label{continuity}
Let $s,r \in \mathbb{N}$ and $\bm{k} \in \mathbb{N}^s$.
Let $\phi \in \Phi_0(r;\bm{k})$ and $\bm{\alpha},\bm{\beta} \in \Delta^r$.
Then
\[
|q(r,\phi,\bm{\alpha})-q(r,\phi,\bm{\beta})| \leq 2\|\bm{\alpha}-\bm{\beta}\|_1 \log_2 s.
\]
\end{proposition}

\begin{proof}
We have that
\begin{align*}
&\quad~ |q(r,\phi,\bm{\alpha})-q(r,\phi,\bm{\beta})|\\
&= \biggl| \sum_{i \in [r]}\alpha_i\sum_{j \in [r]\setminus \lbrace i \rbrace}\alpha_j\log|\phi(ij)| - \sum_{i \in [r]}\beta_i\sum_{j \in [r]\setminus \lbrace i \rbrace}\beta_j\log|\phi(ij)| \biggr|\\
&\leq \biggl| \sum_{i \in [r]}(\alpha_i-\beta_i) \sum_{j \in [r]\setminus \lbrace i \rbrace}\alpha_j\log|\phi(ij)| \biggr| + \biggl| \sum_{j \in [r]}(\alpha_j-\beta_j)\sum_{i \in [r]\setminus \lbrace j \rbrace}\beta_i\log|\phi(ij)| \biggr|\\
&\leq 2\log_2(s) \cdot \|\bm{\alpha}-\bm{\beta}\|_1.
\end{align*}
\hfill\mbox{ }\end{proof}

\subsection{Proof of Lemma~\ref{mainlem}}
Let $\eta > 0$ (assumed without loss of generality to be sufficiently small) and choose an additional constant $\gamma$ so that $0 < \gamma \ll \eta \ll 1/R(\bm{k})$.
By the (standard) Embedding Lemma (see, for example,~\cite[Theorem~2.1]{komsim}), there exist $\eps > 0$ and $m_0 \in \mathbb{N}$ such that the following holds for all $c \in [s]$:
\emph{if $G$ is a graph with a partition $V(G) = W_1 \cup \ldots \cup W_{k_c}$ such that $|W_i| \geq m_0$ for all $i \in [k_c]$ and every pair $(W_i,W_j)$ for $1 \leq i < j \leq k_c$ is $(\eps,\gamma)$-regular, then $K_{k_c} \subseteq G$}.

We may assume that $0 < 1/m_0 \ll \eps \ll \gamma$ since whenever $\eps' \leq \eps$, we have that an $\eps'$-regular pair is also an $\eps$-regular pair.
Let $M$ be the integer returned by Lemma~\ref{multicol} when applied with parameters $\eps^2$ and $s$.
Choose $n_0 \in \mathbb{N}$ and assume, without loss of generality, that $1/n_0 \ll 1/M \ll 1/m_0$.
We have the hierarchy
\begin{equation}\label{eq:constants}
0 < 1/n_0 \ll 1/M \ll 1/m_0 \ll \eps \ll \gamma \ll \eta \ll 1/R(\bm{k}).
\end{equation}
Let $N \geq n_0$ be arbitrary.
Let $H$ be a complete multipartite graph on $N$ vertices with parts $Y_1,\ldots,Y_r$.
We may assume that $r < R(\bm{k})$ otherwise $F(H;\bm{k}) = 0$.
Let $G = (V,E)$ be a graph obtained from $H$ by removing all but one vertex from every part $Y_i$ of size at most $\eta^2 N$ (and all edges incident with the removed vertices).
Write $n := |V|$ and $X_i := Y_i \cap V$ for all $i \in [r]$.
Then $N-n \leq R(\bm{k}) \cdot \eta^2 N$.
So
\[
F(G,\bm{k}) \geq F(H,\bm{k}) \cdot s^{-R(\bm{k}) \eta^2 N^2}
\]
and so
\begin{align}\label{Fchange}
\nonumber \frac{\log_2 F(G;\bm{k})}{n^2/2} &\geq \frac{\log_2 F(G;\bm{k})}{N^2/2} \geq \frac{\log_2 F(H;\bm{k})}{N^2/2} - 3R(\bm{k})\eta^2\log_2 s\\
&\geq \frac{\log_2 F(H;\bm{k})}{N^2/2} - \frac{\eta}{3}.
\end{align}
Define $\bm{\alpha} := (|X_1|/n,\ldots, |X_r|/n)$ and $\bm{\beta} := (|Y_1|/N,\ldots, |Y_r|/N)$.
Then 
\begin{equation}\label{L1diff}
\|\bm{\alpha}-\bm{\beta}\|_1 \leq \frac{R(\bm{k})\eta^2 N}{n} \leq 2R(\bm{k})\eta^2.
\end{equation}
Without loss of generality, there is some $w \in [r]$ such that $X_i = \lbrace x_i \rbrace$ is a singleton for all $i \in [w]$, and $|X_j| > \eta^2 n$ for all $w < j \leq r$.

For the rest of the proof, we will work with $G$ rather than $H$.
Informally, the reason for passing to $G$ is the following.
After applying the Regularity Lemma to $H$ with a valid colouring $\chi$, we do not \emph{a priori} have control on the distribution of coloured edges incident to small parts of $H$.
If the statement of Lemma~\ref{mainlem} asked for a $\phi \in \Phi_0(r;\bm{k})$, we could simply neglect these parts; but since we require $\phi \in \Phi_1(r;\bm{k})$ we cannot do this.
Therefore we introduce $G$ in which each small part $X_i$ is replaced by a token vertex $x_i$, which merely asserts the existence of its part.
But for each $x \in V(G)$, there are only constantly many possible values for $\lbrace \chi(xx_i) : i \in [w] \rbrace$ for all $s$-edge-colourings $\chi$.
Thus we can refine our regularity partition into parts according to these values.
Now we have good control between \emph{all} pairs of parts: if both are large then regularity provides good control; and if one of them is small it is necessarily a single vertex and $\chi$ is constant on all edges between the parts.

Let $\chi : E \rightarrow [s]$ be a $\bm{k}$-valid colouring of $G$.
By the choice of $M$ (that is, by Lemma~\ref{multicol} applied to $G$ and $\chi$ with parameters $\eps^2$ and $s$), there is an (equitable) partition $V = V_1 \cup \ldots \cup V_m$,  with $m \leq M$, which is $\eps^2$-regular simultaneously with respect to all graphs $(V,\chi^{-1}(c))$, $c \in [s]$.

We will now take a common refinement of $X_1,\ldots,X_r$ and $V_1,\ldots,V_m$ which also takes into account attachments to $W := \lbrace x_1,\ldots,x_w \rbrace$.
Namely, for all $j \in [m]$, subdivide $V_j$ into at most $r(s^w+w)$ parts as follows.
Put each vertex in $W \cap V_j$ into a separate part. Now, for any vertices $y,y'$ remaining in $V_j$, put $y$ and $y'$ in the same part if and only if there is some $\ell \in [r]$ such that $\lbrace y,y' \rbrace \subseteq X_\ell$, and $\chi(x_hy) = \chi(x_hy')$ for all $h \in [w]$. 
Thus we obtain a (not necessarily equitable) partition $U_{i,1} \cup \ldots \cup U_{i,m_i}$ of $X_i$ for each $i \in [r]$, where $m_i \leq M(s^w+w)$.
Let $\mathcal{U}$ be the collection of sets~$U_{i,j}$.
It is indexed by 
 \[
 I:=\{\,ij: i \in [r]\mbox{ and } j \in [m_i]\,\}.
 \]

For a colour $c\in [s]$, let $P^c$ consist of all pairs of indices
$\{ig,jh\}\in{I\choose 2}$ such that $\chi^{-1}(c)[U_{i,g},U_{j,h}]$  is $(\eps,\gamma)$-regular, and at least one of the following holds: 
$U_{i,g}$ is a vertex of $W$; $U_{j,h}$ is a vertex of $W$; or $\min \lbrace |U_{i,g}|,|U_{j,h}| \rbrace \geq m_0$.
(So if, say, $U_{i,g}$ is a vertex of $W$, then $\lbrace ig,jh \rbrace \in P^c$ for some $c \in [s]$ since $G[U_{i,g},U_{j,h}]$ is a monochromatic star under~$\chi$.)
We define $E^c\subseteq E$ to be the union of $\chi^{-1}(c)[U_{i,g},U_{j,h}]$ over all pairs $\{ig,jh\}\in P^c$.
Let $
 E_0:=E\setminus (E^1\cup\dots\cup E^s)$. 
 %and call the edges in $E_0$ \emph{inessential}. 
 Thus $E_0$ consists of edges without endpoints in $W$ which are incident with a part of size less than $m_0$; and edges which come from coloured pairs 
that are not $\eps$-regular or have edge density less than~$\gamma$. 
 The following claim, whose proof is fairly standard, shows that $E_0$ cannot
contain many edges.

\begin{claim}\label{badpairs} $|E_0| \leq s\gamma n^2$.
\end{claim}

\begin{claimproof} Call a part $U_{i,g}\subseteq V_\ell$ \emph{small} if
$|U_{i,g}|< \eps|V_\ell|$. Let $E_{\mathrm{small}}\subseteq E$ be the
set of edges that
have at least one vertex in a small part.
Since each $V_\ell$ is subdivided into at most $r(s^w+w)<2R(\bm{k})s^{R(\bm{k})}$ new parts, the number of vertices in small parts is at most $2\eps R(\bm{k})s^{R(\bm{k})}n$
and, trivially, 
 \[
 |E_{\mathrm{small}}|\le 2\eps R(\bm{k})s^{R(\bm{k})}n^2.
 \]

Let $E_{\mathrm{irr}}\subseteq E$ consist of those edges of $G$ that
lie inside some $V_\ell$ or belong to some
colour-$c$ bipartite subgraph $\chi^{-1}(c)[V_\ell,V_{\ell'}]$ which is
not $\eps^2$-regular.
Since $V_1 \cup \ldots \cup V_m$ is an $\eps^2$-regular (equitable) partition, 
we have
 \[
 %\begin{equation}\label{irr}
\left| E_{\mathrm{irr}} \right| \leq  m \binom{\lceil n/m \rceil}2 + s\eps^2  \binom{m}{2} \left\lceil \frac{n}{m} \right\rceil^2% < \eps n^2.
 %\end{equation}
 \]
 which is by $m\ge 1/\eps^2$ at most, say, $\eps n^2$.

Next, we bound the size of $E_0 \setminus (E_{\mathrm{small}}\cup E_{\mathrm{irr}})$. Let $e$ be any edge from this set. 
Since each $U_{i,g}$ is an independent set in $G$, we have
$e \in E(G[U_{i,g},U_{j,h}])$ for some  distinct $ig,jh\in I$. Let $\ell,\ell'\in [m]$ satisfy
$V_\ell\supseteq U_{i,g}$ and $V_{\ell'}\supseteq U_{j,h}$.
Since
$e\not \in E_{\mathrm{small}}$, we have 
 \[
 \min\{\,|U_{i,g}|,\,|U_{j,h}|\,\}\ge \min\{\,\eps|V_\ell|,\,\eps|V_{\ell'}|\,\}\ge \eps\lfloor n/m\rfloor,
 \] 
 which is at least $m_0$ by our choice of constants.
Let $c=\chi(e)$ be the colour of $e$. Since $e\not\in E_{\mathrm{irr}}$, we have that $\ell\not=\ell'$ and $\chi^{-1}(c)[V_\ell,V_{\ell'}]$ is $\eps^2$-regular. Thus Proposition~\ref{badrefine} implies that $\chi^{-1}(c)[U_{i,g},U_{j,h}]$ is $\eps$-regular. Since $e\not \in E^c$,
it must be the case that $\chi^{-1}(c)[U_{i,g},U_{j,h}]\ni e$ has edge density less than $\gamma$. 
We conclude that $E_0 \setminus (E_{\mathrm{small}}\cup E_{\mathrm{irr}})$ has edge density at most $s\gamma$
between any pair $(U_{i,g},U_{j,h})$. Thus
 \[
 |E_0| \leq |E_{\mathrm{small}}| + |E_{\mathrm{irr}}| +  \sum_{\{ig,ih\}\in {I\choose 2}} s\gamma\,|U_{i,g}|\,|U_{j,h}|\leq  2\eps R(\bm{k})s^{R(\bm{k})}n^2+\eps n^2 + s\gamma{n\choose 2} < s\gamma n^2,
\]
proving the claim.
\end{claimproof}

\medskip
\noindent
Define $\phi : \binom{I}{2} \rightarrow 2^{[s]}$ by setting, for all $\lbrace ig,jh\rbrace \in \binom{I}{2}$,
\[
 \phi(ig,jh) := \{c\in [s]: \{ig,jh\}\in P^c\}.
%\{c\in [s] : |U_{i,g}|,|U_{j,h}| \geq m_0 \text{ and } \{ig,jh\}\not\in \Ibad{c}\},
%\left\{ c \in [s]: |U_{i,g}|,|U_{j,h}| \geq m_0 \text{ and } \chi^{-1}(c)[U_{i,g},U_{j,h}] \text{ is } (\eps,\gamma)\text{-regular} \right\}
\]

%Observe that, if $U_{i,g}$ is a set consisting of a vertex of $W$, then $\phi(ig,jh) \neq \emptyset$. 
If neither $U_{i,g}$ nor $U_{j,h}$ is a vertex of $W$ but $\min\{\,|U_{i,g}|,|U_{j,h}|\,\}< m_0$, then $\phi(ig,jh)$ is empty.
Otherwise, $\phi(ig,jh)$ consists of those $c$ for which $\chi^{-1}(c)[U_{i,g},U_{j,h}]$ is $(\eps,\gamma)$-regular. Also, let  $\chi_0 = \chi|_{E_0}$ be the restriction of $\chi$ to $E_0$.

For each $\bm{k}$-valid colouring $\chi$ of $G$, fix one
partition $V=V_1\cup\dots\cup V_{m}$ as above and then define the tuple $(\mathcal{U},I,\phi,E_0,\chi_0)$ accordingly.

\begin{claim}\label{S5}
The number of possible tuples $(\mathcal{U},I,\phi,E_0,\chi_0)$ is at most $2^{\eta n^2/4}$.
\end{claim}

\begin{claimproof}
 Clearly, there are at most $(M(s^w+w))^n \leq (M(s^{R(\bm{k})}+R(\bm{k}))^n < 2^{\eta n^2/12}$ possible partitions of $V$ in which, for all $i \in [r]$, every $x \in X_i$ lies in one of at most $M(s^w+w)$ parts.
Each such partition determines $\mathcal{U}$ and $I$ uniquely (since
the partition $V=X_1\cup\dots\cup X_r$ is fixed throughout the whole proof).

Given $\mathcal{U}$ and $I$, the number of possible $\phi$ is at most
$(2^s)^{\binom{r(Ms^w+w)}{2}} < 2^{\eta n^2/12}$.
By Claim~\ref{badpairs}, the number of ways to choose $E_0$ and colour these edges (i.e. choose $\chi_0$) is, very roughly, at most
\[
\binom{\binom{n}{2}}{s\gamma n^2} (s+1)^{s\gamma n^2} < 2^{\eta n^2/12}.
\]
The claim is proved by multiplying these three bounds.
\end{claimproof}

\medskip
\noindent
Fix a tuple $(\mathcal{U},I,\phi,E_0,\chi_0)$ such that $\mathcal{C} \neq \emptyset$, where $\mathcal{C}$ is the set of colourings $\chi$ which generate it.
Our next step is to provide an upper bound for $|\mathcal{C}|$.
For every $\chi \in \mathcal{C}$, we have $\chi|_{E_0} = \chi_0$. Also, by the definition of $E_0$, every $e \in E \setminus E_0$ lies in some $(\eps,\gamma)$-regular bipartite graph $\chi^{-1}(c)[U_{i,g},U_{j,h}]$ with $c \in [s]$ and $\lbrace ig,jh \rbrace \in \binom{I}{2}$ such that
$\min\{|U_{i,g}|,|U_{j,h}|\}\ge m_0$ or at least one of $U_{i,g},U_{j,h}$ is a vertex of $W$. 
Thus $\{ig,jh\}\in P^c$, that is, $\chi(e) \in \phi(ig,jh)$.
Therefore
\[
|\mathcal{C}| \leq \prod_{ij \in \binom{[r]}{2}}\prod_{\stackrel{gh \in [m_i]\times [m_j]}{\phi(ig,jh) \neq \emptyset}}|\phi(ig,jh)|^{|U_{i,g}|\,|U_{j,h}|}.
\]
Let us agree that $\log_2 0 := 0$.
Then
\begin{equation}\label{S6}
\log_2 |\mathcal{C}| \leq \sum_{ij \in \binom{[r]}{2}}\sum_{gh \in [m_i]\times[m_j]}|U_{i,g}|\,|U_{j,h}| \log_2|\phi(ig,jh)|.
\end{equation}

Let $T := [m_1]\times \ldots \times[m_r]$. We use $T$ to index all `transversal' $r$-tuples of parts from $\mathcal U$,  where we take one part from each of $X_1,\dots,X_r$. For each $\bm{t} = (t_1,\ldots,t_r)$ in $T$, define $\phi_{\bm{t}} : \binom{[r]}{2} \rightarrow 2^{[s]}$ by setting, for $ij\in \binom{[r]}{2}$,
\[
\phi_{\bm{t}}(ij) := \phi(it_i,jt_j).
\]
Recall the definition of $\bm{\alpha}$ after~(\ref{Fchange}).

\begin{claim}\label{numtuple}
$\log_2|\mathcal{C}| \leq (q^*+\sqrt{\gamma})n^2/2$, where
\[
q^* := \max \lbrace q(r,\phi_{\bm{t}},\bm{\alpha}) : (r,\phi_{\bm{t}},\bm{\alpha}) \in \textsc{feas}_1(\bm{k}),\, \bm{t} \in T \rbrace.
\] 
%(In particular, $(r,\phi_{\bm{t}},\bm{\alpha}) \in \feas_1(\bm{k})$ for at least one $\bm{t} \in T$.)
\end{claim}

\begin{claimproof}
We will first show that, for every $c\in [s]$ and $\bm{t} \in T$, the graph $\phi_{\bm{t}}^{-1}(c)$ is $K_{k_c}$-free. 
Indeed, suppose that $i_1, \ldots, i_{k_c}$ span a copy of $K_{k_c}$ in $\phi_{\bm{t}}^{-1}(c)$.
First consider the case when $U_{i_1,t_{i_1}}$ is not a vertex of $W$ but $|U_{i_1,t_{i_1}}| < m_0$.
Then, by the definition of $\phi$, we have that $U_{i_q,t_{i_q}}$ is a vertex of $W$ for all $2 \leq q \leq k_c$.
Moreover, for every $pq \in \binom{[k_c]}{2}$, every edge in $G[U_{i_p,t_{i_p}},U_{i_q,t_{i_q}}]$ is coloured with $c$ by $\chi$, a contradiction.

So, without loss of generality, we may assume that
there is some $0 \leq \ell \leq \min \lbrace k_c,w \rbrace$ such that each of $U_{i_1,t_{i_1}},\ldots,U_{i_\ell,t_{i_\ell}}$ consists of a vertex of $W$ and $|U_{i_q,t_{i_q}}| > m_0$ for all $\ell+1 \leq q \leq k_c$.
Then, by the definition of $\mathcal{U}$, we have that $\chi(e)=c$ for all $e \in G[U_{i_p,t_{i_p}},U_{i_q,t_{i_q}}]$ with $p \in [\ell]$ and $q \in [k_c]\setminus \lbrace p \rbrace$.
By the definition of $P^c\supseteq \phi^{-1}(c)$ and the
Embedding Lemma (that is, our choice of parameters at the beginning of the proof), for all $\ell+1 \leq q \leq k_c$, there is $z_q \in U_{i_q,t_{i_q}}$ such that together these vertices $z_q$ span a copy of $K_{k_c-\ell}$ in  $\chi^{-1}(c)$.
Then $\chi^{-1}(c)$ spans a copy of $K_{k_c}$,
contradicting the $\bm{k}$-validity of~$\chi$. This and the trivial bound  $r < R(\bm{k})$
imply that  $\phi_{\bm{t}} \in \Phi(r;\bm{k})$.
Therefore, for each $\bm{t} \in T$, we have that
$
(r,\phi_{\bm{t}},\bm{\alpha}) \in \feas_0(\bm{k})
$,
and so
\begin{equation}\label{Jeq}
\sum_{ij \in \binom{[r]}{2}}\alpha_i\alpha_j\log_2|\phi(it_i,jt_j)| \leq b(\bm{t}),
\end{equation}
where we define
\[
b(\bm{t}) = \begin{cases} 
q^*/2 &\mbox{if } (r,\phi_{\bm{t}},\bm{\alpha}) \in \feas_1(\bm{k}) \\
r^2\log_2 (s)/2 &\mbox{otherwise} \\
\end{cases}
\]
(i.e.~if $(r,\phi_{\bm{t}},\bm{\alpha}) \notin \feas_1(\bm{k})$ we take a (somewhat arbitrary) trivial bound for $b(\bm{t})$).
The claim will follow from taking a weighted average of~(\ref{Jeq}) by multiplying by $\prod_{\ell \in [r]}|U_{\ell,t_\ell}|$ and summing over all $\bm{t} \in T$.
First consider the right hand side of (\ref{Jeq}).
Let $T_0$ be the set of $\bm{t} \in T$ such that $\phi_{\bm{t}}(ij) = \emptyset$ for some $ij \in \binom{[r]}{2}$.
We will show that the sum of $\prod_{\ell \in [r]}|U_{\ell,t_\ell}|$ over all $\bm{t} \in T\setminus T_0$ is not much less than the sum taken over the whole of $T$.

To this end, fix a pair $\lbrace ig,jh \rbrace \in \binom{I}{2}$ such that $\phi(ig,jh) = \emptyset$.
If at least one edge $e$ in $G[U_{i,g},U_{j,h}]$ is not in $E_0$, then there is some $c \in [s]$ such that $e \in E^c$.
Then $\lbrace ig,jh \rbrace \in P^c$ and so $\phi(ig,jh) \ni c$ is non-empty, a contradiction.
Therefore $E(G[U_{i,g},U_{j,h}]) \subseteq E_0$.
Furthermore, by our definition of $\phi$, we have that $|X_i|,|X_j| \geq \eta^2n$.
Observe that, if one sums only over those $\bm{t} \in T$ that contain $\lbrace ig,jh \rbrace$, then one gets
\[
\sum_{\stackrel{\bm{t} \in T:}{t_i = g, t_j = h}}\prod_{\ell \in [r]}|U_{\ell,t_\ell}| = |U_{i,g}||U_{j,h}| \prod_{\ell \in [r]\setminus \lbrace i,j\rbrace}|X_\ell| \leq \frac{|U_{i,g}||U_{j,h}|}{\eta^4 n^2} \prod_{\ell \in [r]}|X_\ell|.
\]
Then, using the upper bound on $|E_0|$ from Claim~\ref{badpairs}, we have that
\begin{align}\label{T0}
\nonumber \sum_{\bm{t} \in T_0}\prod_{\ell \in [r]}|U_{\ell,t_\ell}| &\leq \sum_{\stackrel{\lbrace ig,jh \rbrace \in \binom{I}{2}:}{E(G[U_{i,g},U_{j,h}]) \subseteq E_0}} \sum_{\stackrel{\bm{t} \in T:}{t_i=g, t_j=h}}\prod_{\ell \in [r]}|U_{\ell,t_\ell}| \leq \frac{|E_0|}{\eta^4 n^2} \prod_{\ell \in [r]}|X_\ell|\\
&\leq \frac{s\gamma}{\eta^4} \prod_{\ell \in [r]}|X_\ell|.
\end{align}
We can now give an upper bound for the weighted average of the right hand of (\ref{Jeq}) as follows:
\begin{align}\label{RHS}
\nonumber \sum_{\bm{t} \in T} \prod_{\ell \in [r]}|U_{\ell,t_\ell}| b(\bm{t}) &\leq \frac{q^*}{2} \sum_{\bm{t} \in T} \prod_{\ell \in [r]}|U_{\ell,t_\ell}|+ \frac{r^2\log_2 s}{2} \sum_{\bm{t} \in T_0}\prod_{\ell \in [r]}|U_{\ell,t_\ell}|\\
 &\stackrel{(\ref{T0})}{\leq}  \prod_{\ell \in [r]}|X_\ell| \left( \frac{q^*}{2} + \frac{r^2 s\gamma \log_2 s}{2\eta^4} \right) \leq \prod_{\ell \in [r]}|X_\ell| \frac{q^* + \sqrt{\gamma}}{2}.
\end{align}
Using this bound together with a weighted average of the left hand side of~(\ref{Jeq}), we have that
\begin{align*}
&\quad~ \frac{q^* + \sqrt{\gamma}}{2} \prod_{\ell \in [r]}|X_\ell|\\
&\geq \sum_{\bm{t} \in T}   
 %\left( 
 \sum_{ij \in \binom{[r]}{2}}\alpha_i\alpha_j\log_2|\phi(it_i,jt_j)| 
 %\left( 
 \prod_{\ell \in [r]}|U_{\ell,t_\ell}| 
 %\right)
 %\right)
 \\
&= \sum_{ij \in \binom{[r]}{2}}\alpha_i\alpha_j
 %\left( 
\sum_{gh \in [m_i]\times[m_j]}|U_{i,g}|\,|U_{j,h}|\log_2|\phi(ig,jh)|  
 %\left( 
\sum_{\bm{t} \in T :\atop t_i = g,t_j=h}\ \prod_{\ell \in [r]\setminus\{i,j\}}|U_{\ell,t_\ell}|
 % \right)\right)
\\
&= \sum_{ij \in \binom{[r]}{2}} \frac{|X_i|}{n} \cdot \frac{|X_{j}|}{n} 
 %\left( 
\sum_{gh \in [m_i]\times[m_j]}|U_{i,g}|\,|U_{j,h}|\log_2|\phi(ig,jh)|  
 %\left(
 \prod_{ \ell \in [r]\setminus \lbrace i,j\rbrace}|X_\ell| 
 %\right) \right)
\\
&\stackrel{\mathclap{(\ref{S6})}}{\geq} \frac{1}{n^2}  \log_2 |\mathcal{C}| \prod_{ \ell \in [r]}|X_\ell|,
\end{align*}
proving Claim~\ref{numtuple}.
\end{claimproof}

\medskip
\noindent
Let $\bm{t}^* \in T$ be such that $q^* = q(r,\phi_{\bm{t}^*},\bm{\alpha})$.
Recall that $\bm{\beta} = (|Y_1|/N,\ldots,|Y_r|/N)$.
Then $(r,\phi_{\bm{t}^*},\bm{\alpha})$ and hence $(r,\phi_{\bm{t}^*},\bm{\beta})$ lies in $\feas_1(\bm{k})$. 
Now Claims~\ref{S5} and~\ref{numtuple} and Proposition~\ref{continuity} imply that
\begin{eqnarray*}
\frac{\log_2 F(H;\bm{k})}{N^2/2} &\stackrel{(\ref{Fchange})}{\leq}& \frac{\log_2 F(G;\bm{k})}{n^2/2} + \frac{\eta}{3} \leq \frac{5\eta}{6} + q^* + \sqrt{\gamma} < q(r,\phi_{\bm{t}^*},\bm{\alpha}) + \frac{6\eta}{7}\\
\nonumber &\leq& q(r,\phi_{\bm{t}^*},\bm{\beta}) + 2\log_2 (s) \|\bm{\alpha}-\bm{\beta}\|_1 + \frac{6\eta}{7} \stackrel{(\ref{L1diff})}{\leq} q(r,\phi_{\bm{t}^*},\bm{\beta}) + \eta,
\end{eqnarray*}
completing the proof of the lemma.\hfill$\square$

\section{Proofs of Theorems~\ref{optred} and~\ref{th:PartiteStructure}}\label{OptredProof}

%We will now use Lemma~\ref{mainlem} to prove Theorems~\ref{optred} and~\ref{th:PartiteStructure}.

\subsection{Proof of Theorem~\ref{optred}}
By Lemma~\ref{lm:lowerbd}, it suffices to show that for every $\eta > 0$, there exists $n_0 \in \mathbb{N}$ such that $\log_2 F(n;\bm{k}) \leq (Q(\bm{k})+\eta)n^2/2$ for all $n \geq n_0$.
Fix $\eta > 0$ and obtain $n_0$ from Lemma~\ref{mainlem}.
Now let $n \geq n_0$.
By Theorem~\ref{zykov}, there exists a complete multipartite graph $G$ on $n$ vertices with $F(G;\bm{k}) = F(n;\bm{k})$.
The required upper bound on $\log_2 F(G;\bm{k})$ follows immediately from Lemma~\ref{mainlem}.\hfill$\square$

\subsection{Proof of Theorem~\ref{th:PartiteStructure}}
Suppose that there is $\delta > 0$ which contradicts the claim.
We need the following claim, which uses a compactness argument to show that a triple in $\feas_1(\bm{k})$ which is almost optimal is in fact `close' to a $Q_1$-optimal triple.
%which implies that, to obtain a contradiction, it suffices to find a $\phi \in \Phi(r;\bm{k})$ such that $|\phi(ij)| \geq 1$ for all $ij \in \binom{[r]}{2}$, such that $q(r,\phi,\bm{\alpha})$ is very close to $Q(\bm{k})$.

\begin{claim}\label{compact}
There exists $\eta > 0$ such that for all $(r,\phi,\bm{\alpha}) \in \feas_1(\bm{k})$ with $q(r,\phi,\bm{\alpha}) \geq Q(\bm{k})-2\eta$, there is a $Q_1$-optimal triple $(r,\phi,\bm{\alpha}')$ such that $\|\bm{\alpha}' - \bm{\alpha}\|_1 \leq \delta$.
\end{claim}

\begin{claimproof}
Suppose this is not the case.
Then for all $n \in \mathbb{N}$, there exists $(r,\phi,\bm{\alpha}_n) \in \feas_1(\bm{k})$ with
\begin{equation}\label{limitQ}
q(\phi,\bm{\alpha}_n) \geq Q(\bm{k}) - \frac{1}{n},
\end{equation}
but for all $\bm{\alpha}'_n \in \Delta^{r}$ with $\|\bm{\alpha}_n-\bm{\alpha}'_n\|_1 < \delta$, we have that $(r,\phi,\bm{\alpha}'_n)$ is not $Q_1$-optimal.

Consider the sequence $(\bm{\alpha}_1,\bm{\alpha}_2\ldots)$.
Since $\Delta^r$ is closed and bounded, the Heine-Borel theorem implies that it is compact.
Therefore there is some subsequence $(\bm{\alpha}_{n_1},\bm{\alpha}_{n_2},\ldots)$ of $(\bm{\alpha}_1,\bm{\alpha}_2,\ldots)$ which converges (in any norm, since $r$ is finite).
Let $\bm{\lambda} := \lim_{k \rightarrow \infty}\bm{\alpha}_{n_k}$.
Observe that $\bm{\lambda} \in \Delta^r$, so $(r,\phi,\bm{\lambda}) \in \feas_1(\bm{k})$.
Having fixed $r,\phi$, observe that $q(r,\phi,\bm{\lambda}) = 2\sum_{ij \in \binom{[r]}{2}}\lambda_i\lambda_j\log|\phi(ij)|$ is a continuous function of $\bm{\lambda}$.
Therefore
\[
\lim_{k \rightarrow \infty} q(r,\phi,\bm{\alpha}_{n_k}) = q(r,\phi,\bm{\lambda}).
\]
Together with~(\ref{limitQ}), this implies that $q(r,\phi,\bm{\lambda}) = Q(\bm{k})$, and so $(r,\phi,\bm{\lambda})$ is $Q_1$-optimal.
Now, since $\bm{\alpha}_{n_k} \rightarrow \bm{\lambda}$, we can choose $N \in \mathbb{N}$ such that $\|\bm{\alpha}_N-\bm{\lambda}\|_1 < \delta$.
This contradicts our assumption and hence proves the claim.
\end{claimproof}

\medskip
\noindent
Choose $\eta$ as in the claim.
Obtain $n_0 \in \mathbb{N}$ by applying Lemma~\ref{mainlem} with $\eta$.
Since we supposed that $\delta > 0$ contradicts the statement of Theorem~\ref{th:PartiteStructure}, there exists a complete multipartite graph $G$ on $n \geq n_0$ vertices such that $F(G;\bm{k}) \geq 2^{(Q(\bm{k})-\eta)n^2/2}$ and $G$ is a counterexample to the statement.
Let $V_1,\ldots,V_r$ be the parts of $G$ and define $\bm{\alpha} := (|V_1|/n,\ldots,|V_r|/n)$.
Then, for all $Q_1$-optimal triples $(r,\phi,\bm{\alpha}')$, we have that $\|\bm{\alpha}-\bm{\alpha}'\|_1 > \delta$.
Lemma~\ref{mainlem} and our assumption on $G$ imply that there exists $\phi \in \Phi_1(r;\bm{k})$ such that
\begin{equation}\label{Feq}
Q(\bm{k})-\eta \leq \frac{\log_2 F(G;\bm{k})}{n^2/2} \leq q(r,\phi,\bm{\alpha}) + \eta.
\end{equation}
Claim~\ref{compact} immediately gives a contradiction, completing the proof of the Theorem~\ref{th:PartiteStructure}.\hfill$\square$

\section{Concluding remarks}\label{conclusion}

The referee of this paper asked if the cases where $F(\bm{k})$ was determined in~\cite{abks} can be done using our optimisation problem. While the answer is in the affirmative, some claims from~\cite{abks} are more conveniently derived by working with graphs rather than feasible solutions. 
For example, following~\cite{abks} let us show that
 \begin{equation}\label{eq:sumLDL}
\sum_{\ell \in [s]} \ell d_\ell \leq s\left(1-\frac{1}{k-1}\right),
 \end{equation}
 where  $\bm{k} := (k,\ldots,k)$ has length~$s$, $(r,\phi,\bm{\alpha}) \in \feas_0(\bm{k})$ is an
arbitrary feasible solution, and  we define
$$
d_\ell := 2\sum_{\stackrel{ij \in \binom{[r]}{2}}{|\phi(ij)|=\ell}}\alpha_i\alpha_j,\quad \mbox{for $ß\ell\in [s]$}.
$$
%Then $q(\phi,\bm{\alpha}) = \sum_{\ell \in [s]}d_\ell \log_2 \ell$.
 The shortest way is probably to consider the graph $G_{\phi,\bm{\alpha}}(n)$ from the proof of Lemma~\ref{lm:lowerbd}. 
For $c\in [s]$, let $H_c$ be the subgraph of $G_{\phi,\bm{\alpha}}(n)$ spanned by pairs of parts $(X_i,X_j)$ such that $c \in \phi(ij)$.
Then $H_c$ is $K_k$-free for all colours $c \in [s]$ and so Tur\'an's theorem implies that $e(H_c) \leq (1-\frac{1}{k-1})n^2/2$.
Thus we have that,  as $n \rightarrow \infty$,
$$
\sum_{\ell \in [s]} \ell d_\ell = 2\sum_{c \in [s]}\sum_{\stackrel{ij \in \binom{[r]}{2}}{c\in\phi(ij)}}\alpha_i\alpha_j = 2\sum_{c \in [s]}\frac{e(H_c)+O(n)}{n^2} \leq s\left(1-\frac{1}{k-1}\right)+o(1),
$$
 which gives the claimed inequality~\eqref{eq:sumLDL}. Interestingly,~\eqref{eq:sumLDL} and  the trivial
constraints $d_\ell\ge 0$ for $\ell\in [s]$ imply the sharp upper bound on $q(\phi,\bm{\alpha}) =\sum_{\ell=1}^s d_\ell \log_2\ell$ when $s \in \lbrace 2,3\rbrace$ and when $\bm{k}=(4,4,4,4)$. (If $\bm{k}=(3,3,3,3)$, then an additional constraint, analogous to~\eqref{eq:sumLDL}, suffices to determine $Q(\bm{k})$, see~\cite{abks}.)

Unfortunately, the problem of (numerically) solving Problem~$Q_2$ seems
rather difficult even for moderately small
$\bm{k}$. If we have a candidate pair $(r,\phi)$, then the 
Lagrange Multiplier Method gives a linear program which either returns a best
possible $\bm{\alpha}$ for this $(r,\phi)$ in the interior of $\Delta^r$, or
implies that there is an optimal solution on the boundary so we can reduce $r$ by one.
This calculation
can be efficiently implemented. However, the number of possible pairs $(r,\phi)$ becomes large very quickly. Here, the quest of replacing the crude bound $r<R(\bm{k})$ by a better one leads to the following Ramsey-type question. Namely, $r$ can be bounded by $R_2(\bm{k})-1$, where
we define $R_2(\bm{k})$ to be the smallest $r$ such that for every choice of a
\emph{list-colouring} $\phi: {[r]\choose 2}\to {[s]\choose 2}$
there is $c\in [s]$ with $\phi^{-1}(c)$ containing a $k_c$-clique.
Clearly, the definition would not change if we restrict ourselves to
lists of size at least $2$, so we can assume $r< R_2(\bm{k})$ in the statement
of Problem~$Q_2$.
The problem of estimating $R_2(\bm{k})$ runs into similar difficulties as those for the classical version $R(\bm{k})$. 
It is a special case of a parameter studied in~\cite{setcol}, and seems to grow fast.
For example, in~\cite{setcol} it was shown that $R_2(5,5,5) \geq 20$, which is already too large for a na\"{i}ve enumeration of feasible $\phi$ by computer.

As we mentioned, the existence of the limit in~\eqref{eq:lim} can be shown
by an easy modification of the proof for the case $k_1=\dots=k_s$ in~\cite{abks}. In fact, there are two different proofs. 
The one that appears in the published version of~\cite{abks} was suggested
by an anonymous referee and uses an entropy inequality of Shearer
to show that $\log F(n;\bm{k})/n^2$ is a non-increasing function of~$n$. 

The other proof, which was the original argument by Alon et al~\cite{abks}, is
similar to our proof of Theorem~\ref{optred}. In our language, it can be sketched as follows. Fix a large $N$ such that $\log_2 F(N;\bm{k})/{N\choose 2}$ is close to the limit superior of~\eqref{eq:lim}. Take an $\eps$-regular partition $V(G)=V_1\cup \dots\cup V_m$ of an
arbitrary $\bm{k}$-extremal order-$N$ graph $G$ with a `typical' colouring $\chi$. Let $\phi(ij)$ be
the set of those colours $c\in[s]$ for which $\chi^{-1}(c)[V_i,V_j]$ is an $(\eps,\gamma)$-regular pair. As in Lemma~\ref{lm:lowerbd}, use this function $\phi:{[m]\choose 2}\to 2^{[s]}$ with the uniform vector $\bm{\alpha}=(1/m,\dots,1/m)$ to produce
graphs of order $n\to\infty$ with at least $2^{q(m,\phi,\bm{\alpha})n^2/2-O(n)}$
valid colourings. Since $q(m,\phi,\bm{\alpha})$ can be made arbitrarily close to the limit superior of~\eqref{eq:lim} by choosing small $\gamma\gg\eps\gg1/N$, the limit in~\eqref{eq:lim} exists.

The latter proof can be adopted to prove~Theorem~\ref{optred}
(by applying symmetrisation to reduce the triple $(m,\phi,\bm{\alpha})$
to one with fewer than $R(\bm{k})$ parts). 
 However, our proof (where the Regularity Lemma is applied 
after the symmetrisation) has the advantages of giving some explicit
(although rather bad) bound on the rate of convergence in~\eqref{eq:lim} and implying Theorem~\ref{th:PartiteStructure}
as well.

%Theorem~\ref{optred} shows that the maximum value $Q(\bm{k})$ of Problem~($\star$) asymptotically determines $\log F(n;\bm{k})$ for large $n$.

%We plan to further address the `stability' of $\bm{k}$-extremal graphs in a forthcoming work; that is, to describe the asymptotic structure of every graph $G$ on $n$ vertices (as well as typical $\bm{k}$-valid colourings of $G$) with $F(G;\bm{k})  = F(n,\bm{k}) \cdot 2^{o(n^2)}$ from the set of optimal solutions to $Q(\bm{k})$.
%Unfortunately, the proofs are long and technical so we do not present them in this paper. 

Despite Theorem~\ref{th:PartiteStructure}, there may be order-$n$ graphs $G$ with $F(G;\bm{k})=2^{(Q(\bm{k})+o(1))n^2/2}$ which are very far in edit distance from being complete multipartite. For example, if $\bm{k}=(4,3)$, 
%(when $Q(\bm{k})=1/2$ is given by the unique feasible vector corresponding to
%$K_{n/2,n/2}$), 
then one can take for $G$ an equitable complete bipartite 
graph with parts $A\cup B$ and add any triangle-free graph into $A$
(e.g.\ a blow-up of a pentagon which is far from being complete partite). Here, we can colour edges between $A$ and $B$ arbitrarily provided all edges inside $A$ have colour $1$. Thus
$F(G;(4,3))\ge 2^{|A|\,|B|} = 2^{\frac{1}{2}\binom{n}{2}+O(n)}$, while $Q((4,3))$ is easily seen to be equal to $1/2$.

Interestingly, our follow-up results (in preparation) show that all $(4,3)$-extremal graphs of sufficiently large order $n$ happen to be in fact 3-partite. For example, if $n=2m+1$ is odd (and large), then the unique extremal graph is $K_{m,m-1,2}$. In order to illustrate how a small part 
can increase the number of colourings, let us show that 
 \begin{equation}\label{eq:43}
 F(K_{m,m,1};(4,3))\ge 2\cdot 2^{m(m+1)}-2^{m^2},
 \end{equation}
 that is, the number of $(4,3)$-valid colourings of $H := K_{m,m,1}$ is by factor $2-o(1)$
larger than that for the Tur\'an graph $K_{m+1,m}$.
If $H$ has parts $V_1\cup V_2\cup V_3$ with $|V_3|=1$, then $H$ has 
$2^{m(m+1)}$ colourings where $G[V_1\cup V_3,V_2]$ is coloured arbitrarily 
while all edges between $V_1$ and $V_3$ have colour $1$. Similarly we have 
$2^{m(m+1)}$ colourings where $V_3$ is `bundled' with $V_2$ (and all edges 
between $V_2$ and $V_3$ get colour $1$). All colourings that appear twice
are exactly those that assign colour 1 to all edges incident to $V_3$, so there are $2^{|V_1|\,|V_2|}=2^{m^2}$ of them, giving~\eqref{eq:43}.

The above example shows that one can have parts of size $o(n)$ in Theorem~\ref{th:PartiteStructure} even for $\bm{k}$-extremal graphs. (These parts will correspond to zero entries of $\bm{\alpha}$ in the limit.)
%Theorem~\ref{zykov} shows that for every $\bm{k}$ and $n$, there is a complete multipartite graph of order $n$ which is $\bm{k}$-extremal. 
Nonetheless, we conjecture that Theorem~\ref{zykov} captures all extremal graphs:

\begin{conjecture}\label{cj:extr}
For every $n,s \in \mathbb{N}$ and $\bm{k} \in \mathbb{N}^s$, every $n$-vertex $\bm{k}$-extremal graph is complete multipartite.
\end{conjecture}
 
In a future paper, we hope to provide a sufficient condition for this to be true for all $n\ge n_0(\bm{k})$ and apply the developed theory to solving the problem for new values of~$\bm{k}$. 
We note that, in the different setting of forbidden cliques with prescribed colour patterns explored in~\cite{bhs}, the corresponding version of Conjecture~\ref{cj:extr} holds in some cases.

\end{document}